\documentclass[10pt, leqno]{article}

\usepackage[utf8x]{inputenc}
\usepackage{amsfonts}
\usepackage{amsmath}
\usepackage{amssymb}
\usepackage{amsthm}
\usepackage{mathrsfs}
\usepackage{latexsym}
\usepackage{graphicx}
\usepackage{bm}
\usepackage{inputenc}
\usepackage{enumerate}
\usepackage{enumitem}
\usepackage{hyperref}
\usepackage{graphicx}

\swapnumbers

\newtheorem{Lemma}{Lemma}[section]
\newtheorem{Corollary}[Lemma]{Corollary}
\newtheorem{Theorem}[Lemma]{Theorem}

\theoremstyle{definition}
\newtheorem{Miniremark}[Lemma]{}
\newtheorem{Definition}[Lemma]{Definition}

\theoremstyle{remark}
\newtheorem{Remark}[Lemma]{Remark}

\newtheoremstyle{proof*}
{3pt}
{3pt}
{\rmfamily}
{}
{\bfseries}
{.}
{.5em}
{\thmnote{#3}}
\theoremstyle{proof*}
\newtheorem*{proof*}{}

\DeclareMathOperator{\restrict}{\llcorner}

\DeclareMathOperator{\Tan}{Tan}     
\DeclareMathOperator{\spt}{spt}     
\DeclareMathOperator{\im}{im}       
\DeclareMathOperator{\Lip}{Lip}     
  
\DeclareMathOperator{\dmn}{dmn}     
\DeclareMathOperator{\Nor}{Nor}
     
\DeclareMathOperator{\Der}{D}       
       
\DeclareMathOperator{\ap}{ap}  

\DeclareMathOperator{\trace}{trace}

\DeclareMathOperator{\Dis}{Dis}
\DeclareMathOperator{\discr}{discr}

\newcommand{\Real}[1]{ \mathbf{R}^{#1}}
\newcommand{\Haus}[1]{ \mathscr{H}^{#1} }
\newcommand{\Leb}[1]{ \mathscr{L}^{#1} }
\newcommand{\rect}[1]{(\mathscr{H}^{#1},#1)}
\newcommand{\Hdensity}[3]{\bm{\Theta}^{#1}(\mathscr{H}^{#1}\restrict \,#2,#3 )}

\title{Normal bundle and Almgren's geometric inequality for singular varieties of bounded mean curvature}
\author{Mario Santilli}

\begin{document}
	\maketitle
	
	\begin{abstract}
		In this paper we deal with singular varieties of bounded mean curvature in the viscosity sense. They contain all varifolds of bounded generalized mean curvature. In the first part we investigate the second-order properties of these varieties, obtaining results that are new also in the varifold's setting. In particular we prove that the generalized normal bundle of these varieties satisfies a natural Lusin (N) condition, which allows to extend the classical Coarea formula for the Gauss map of smooth varieties, and to introduce for all integral varifolds of bounded mean curvature a natural definition of second fundamental form, whose trace equals the generalized varifold mean curvature. In the second part, we use this machinery to extend a sharp geometric inequality of Almgren to all compact varieties of bounded mean curvature in the viscosity sense and we characterize the equality case. As a consequence we formulate sufficient conditions to conclude that the area-blow-up set is empty for sequences of varifolds whose first variation is controlled.
	\end{abstract}

\paragraph{\small MSC-classes 2010.}{\small 49Q20, 49Q10, 53A07, 53C24, 35D40.}

\section{Introduction}

\paragraph*{General setting.} In this paper we deal with the following class of singular varieties.

\begin{Definition}\label{definition of (m,h)}(see \cite[2.1]{MR3466806}\footnote{This definition is equivalent to \cite[2.1]{MR3466806} by \cite[8.1]{MR3466806}.})
	Suppose $ 1 \leq m < n $ are integers, $ \Omega $ is an open subset of $ \Real{n} $, $ \Gamma $ is relatively closed in $ \Omega $ and $ h \geq 0 $. We say that $ \Gamma $ is an \emph{$ (m,h) $ subset of $ \Omega $} provided it has the following property: if $ x \in \Gamma $ and $ f $ is a $ \mathscr{C}^{2} $  function in a neighbourhood of $ x $ such that $ f|\Gamma $ has a local maximum at $ x $ and $ \nabla f(x) \neq 0 $, then
	\begin{equation*}
	\trace_{m}\Der^{2}f(x) \leq h |\nabla f(x)|,
	\end{equation*}
	where $ \trace_{m} \Der^{2}f(x) $ is the sum of the lowest $ m $ eigenvalues of $ \Der^{2}f(x) $.
\end{Definition}
The $ (m,h) $ sets can be roughly described as "varieties with mean curvature bounded by $ h $ in the viscosity sense". They were introduced by Brian White in \cite{MR3466806} to study the area-blow-up of sequences of submanifolds (or varifolds) and they contain all $ m $ dimensional varifolds $ V $ such that $ \|\delta V \| \leq h \|V\|$, see \cite[2.8]{MR3466806}\footnote{In this paper we adopt the terminology in \cite[Appendix C]{MR855173} for varifolds; in particular note that the variation function $ \mathbf{h}(V, \cdot) $ (i.e.\ generalized mean curvature of $ V $) differs from the one adopted in Allard's paper \cite[4.2]{MR0307015} by a sign.}. Similar notions have been considered in the theory of viscosity solutions of PDE's; see \cite{MR1190161}, \cite{MR3695374} and \cite{MR3823884}.

In \cite{Santilli2019} we have systematically investigated a notion of curvature for arbitrary closed sets. The first goal of the present paper is to employ this machinery to analyze the pointwise curvature properties of $(m,h)$ sets. This is in analogy with the study of the second-order pointwise differentiability of viscosity solutions of PDE's, see \cite{MR995142}. Our investigation starts with the following definition (see also \cite{MR534512} and \cite{MR2031455}). If $ A \subseteq \Real{n} $ is closed we define the \emph{generalized unit normal bundle of $ A $} as
\begin{equation*}
N(A) = (A \times \Real{n}) \cap \{ (a,u) : |u|=1, \;  \bm{\delta}_{A}(a+su)=s \; \textrm{for some $ s > 0 $}\}
\end{equation*}
(here $ \bm{\delta}_{A} $ is the distance function from $ A $). Notice that $N(A)$ is the natural generalization to our geometric setting of the second order super-differential of a function considered in \cite{MR995142}. The set $ N(A) $ is always a countably $n-1$ rectifiable subset of $ \Real{n} \times \Real{n}$ (in the sense of \cite[3.2.14]{MR0257325}) and an appropriate notion of second fundamental form 
\begin{equation*}
 Q_{A}(a,u): T_{A}(a,u) \times T_{A}(a,u) \rightarrow \Real{},
\end{equation*}
where $T_{A}(a,u) $ is a linear subspace of $ \Real{n} $, exists at $ \Haus{n-1} $ almost all $ (a,u) \in N(A) $ (see \ref{curvature of arbitrary closed sets 3}). For an arbitrary closed set $ A $ the dimension of $T_{A}$ may vary from point to point. One of the main result of the paper (see \ref{Lusin property for (m,h) sets}) shows that if $ A $ is an $ (m,h) $ subset of $ \Real{n} $ then the normal bundle $N(A)$ satisfies the following remarkable Lusin (N) condition, provided that $A$ is a countable union of sets of finite $ \Haus{m} $ measure. 

\begin{Definition}\label{Lusin Property intro}
	Suppose $ A \subseteq \Real{n} $ is a closed set, $ \Omega \subseteq \Real{n} $ is an open set and $ 1 \leq m < n $ is an integer. We say that $ N(A) $ satisfies the \textit{$ m $ dimensional Lusin (N) condition in $ \Omega $} if and only if the following property holds:
	\begin{equation*}
	\Haus{n-1}(N(A)\cap \{(a,u) : a \in Z  \})=0
	\end{equation*}
	for every $Z \subseteq  A \cap \Omega$ with $\Haus{m}(A^{(m)} \cap Z) =0 $. Here $A^{(m)} $ is the set of points where $ A $ can be touched by a ball from $ n-m $ linearly independent directions, (see \ref{curvature of arbitrary closed sets 4}).
\end{Definition}

It follows from a recent result of Schneider \cite{MR3272763} that a typical (in the sense of Baire category) compact convex hypersurface in $ \Real{n} $ is $ \mathcal{C}^{1} $ but does not possess the $n-1$ dimensional Lusin (N) condition. The validity of this condition (which is new also in the varifold case) is a consequence of the weak maximum principle, which is the defining property of $(m,h)$ sets. This condition has deep consequences on the curvature properties of these varieties. For example it implies that the first $ m $ principal curvatures of an $(m,h)$ set  are finite; this is in sharp contrast with the typical behavior of a convex surface; see \cite{MR3272763} and \cite[6.3]{Santilli2019}.

This good curvature-behavior allows to extend the Coarea formula for the generalized Gauss map. If $ M $ is an $ m $ dimensional $ \mathcal{C}^{2} $ submanifold of $ \Real{n} $ without boundary, $ N(M) $ is the unit normal bundle and $ Q_{M} $ is the second fundamental form then the area of the generalized Gauss map of $ M $ can be expressed in terms of the curvature of $ M $ in the following way: if $B $ is an $ \Haus{n-1} $ measurable subset of $ N(M) $ then
	\begin{flalign}\label{intro: coarea formula smooth}
& \int_{\mathbf{S}^{n-1}} \Haus{0}\{ a: (a,u)\in B  \}\,d\Haus{n-1}u \notag \\
& \qquad \qquad \qquad   = \int_{M}\int_{\{ \eta : (z, \eta)\in B  \} } | \discr Q_{M}(z,\zeta)| d\Haus{n-m-1}(\zeta) \, d\Haus{m}(z),
\end{flalign}
where $ \discr Q_{M}(z, \zeta) $ is the discriminant of the symmetric bilinear form $ Q_{M}(z, \zeta) $, see \cite[1.7.10]{MR0257325}. Smoothness of $ M $ readily reduces the proof of this result to an application of classical Coarea formula. From a slightly different point of view we could say that the smoothness of $ M $ readily implies the Lusin (N) condition, which in turn implies the validity of the Coarea formula. For our singular varieties we may use the Lusin (N) condition to obtain such a formula following the same argument. Summarizing the results mentioned so far we state the first main result of the paper.

\begin{Theorem}[Coarea formula for the spherical image map of $(m,h)$ sets]\label{intro: area formula}
	Suppose $ 1 \leq m \leq n-1 $, $ 0 \leq h < \infty $, $\Gamma$ is an $ (m,h) $ subset of $ \Real{n} $ that is a countable union of sets with finite $ \Haus{m} $ measure. Then $ N(\Gamma) $ satisfies the $ m $ dimensional Lusin (N) condition and
	\begin{flalign*}
	& \int_{\mathbf{S}^{n-1}} \Haus{0}\{ a: (a,u)\in B  \}\,d\Haus{n-1}u \\
	& \qquad \qquad \qquad   = \int_{\Gamma}\int_{\{ \eta : (z, \eta)\in B  \} } | \discr Q_{\Gamma}| d\Haus{n-m-1} d\Haus{m}z.
	\end{flalign*}
	whenever $ B \subseteq N(\Gamma) $ is $ \Haus{n-1} $ measurable. Moreover,
	\begin{equation*}
	\dim T_{\Gamma}(z,\eta) = m,  \qquad	\trace Q_{\Gamma}(z, \eta) \leq h
	\end{equation*}
	for $ \Haus{n-1} $ a.e.\ $ (z, \eta) \in N(\Gamma) $.
	\end{Theorem}

 
Theorem \ref{intro: area formula} clearly shows that $ Q_{\Gamma} $ and $ \trace Q_{\Gamma} $ naturally describe key geometric properties of general $(m,h)$ sets, thus providing natural notions of second fundamental form and mean curvature for this class of varieties. In case of \emph{integral varifolds} we also prove the agreement of the trace of the second fundamental form with the generalized mean curvature. The restriction to \emph{integral} varifolds is technical and only due to the fact that we rely on the locality theorem of Sch\"atzle \cite[4.2]{MR2472179}, which is currently not available for non-integral varifolds.

\begin{Corollary}[Second fundamental for integral varifolds of bounded mean curvature]\label{intro: locality}
Suppose $ 1 \leq m \leq n-1 $, $ V\in \mathbf{V}_{m}(\Real{n}) $ is an integral varifold such that $\|\delta V\| \leq c \|V\| $ for some $ 0 \leq c < \infty $. Then
\begin{equation*}
\trace Q_{\spt \|V\|}(z, \eta) = \mathbf{h}(V,z) \bullet \eta \quad \textrm{for $ \Haus{n-1} $ a.e.\ $ (z,\eta) \in N(\spt \|V\|) $.}
\end{equation*}
\end{Corollary}

Combining \ref{intro: area formula} and \ref{intro: locality} we obtain new insights in the study of the curvature properties of varifolds. Besides the classical work on \emph{curvature varifolds} in \cite{MR825628} and \cite{MR1412686}, another recent contribution in this field is the proof of the second-order-rectifiability for varifolds: in \cite{zbMATH06157228} (see also \cite{MR2064971}-\cite{MR2472179}) for \emph{integral} varifolds with locally bounded first variation and in \cite{2019arXiv190702792S} for \emph{rectifiable} varifolds with a uniform lower bound on the density and bounded generalized mean curvature.

The other main contribution of this paper is the extension of Almgren's geometric inequality to compact $(m,h)$ sets. 

\begin{Theorem}\label{intro: sphere theorem}
If $ 1 \leq m \leq n-1 $, $ h > 0 $ and $ \Gamma $ is a non-empty compact $(m,h)$ subset of $ \Real{n} $ then 
\begin{equation*}
\Haus{m}(\Gamma) \geq \Big(\frac{m}{h}\Big)^{m} \Haus{m}(\mathbf{S}^{m}).
\end{equation*}
	
Moreover if the equality holds and $ \Gamma = \spt (\Haus{m}\restrict \Gamma) $ then there exists an $m+1$ dimensional plane $ T $ and $ a \in \Real{n} $ such that
\begin{equation*}
\Gamma = \partial \mathbf{B}(a, m/h) \cap T.
\end{equation*}
\end{Theorem}

If $ \Gamma $ is the support of a rectifiable varifold $ V $ with a uniform lower bound on the density such that $ \| \delta V \| \leq m \| V \| $ then this theorem is contained in \cite{MR855173}. Our proof generalizes Almgren's method to $(m,h)$ sets and combines it with the novel facts stated in \ref{intro: area formula}, which are new also in the varifold's setting. As a byproduct our proof somewhat simplifies several steps of Almgren's original argument for varifolds. We now briefly describe the main steps of the proof. Firstly we can suitably rescale $ \Gamma $ to have $ m = h $. For the inequality case we use compactness of $ \Gamma $ to see that for each $ \eta \in \mathbf{S}^{n-1} $ there exists an $ (n-1)$ dimensional plane $ \pi $ perpendicular to $ \eta $ such that $ \Gamma $ lies on one side of $ \pi $ and touches $ \pi $ at least in one point. This can be precisely stated saying that the projection onto $ \mathbf{S}^{n-1} $ of the contact set 
\begin{equation*}
	C = (\Gamma \times \mathbf{S}^{n-1}) \cap  \{ (z, \eta) : (w-z)\bullet \eta \leq 0 \; \textrm{for every $ w \in \Gamma $} \} \subseteq N(\Gamma)
\end{equation*}
equals $ \mathbf{S}^{n-1} $. Then the estimate $ \trace Q_{\Gamma} \leq m $ in \ref{intro: area formula} and the more elementary fact that $ Q_{\Gamma} $ has a sign when restricted on $ C $, allows to obtain $ \Haus{m}(\Gamma)\geq \Haus{m}(\mathbf{S}^{m}) $. This crucial quantitative estimate is obtained working directly on the projection of the contact set $ C $ of $ \Gamma $, combining the Coarea formula \ref{intro: area formula} and the Barrier principle of White \cite[7.1]{MR3466806}, and with no structural or smoothness assumptions at the touching points. This argument originates from the approach to Almgren's theorem developed in \cite{Men12} and it is somewhat more direct than Almgren's method, which instead uses the convex hull of $ \Gamma $. Moving to the proof of the equality case, we first combine \cite[3.2]{MR3466806} with the Strong Barrier principle in \cite[7.3]{MR3466806} to conclude that at each point of $ \Gamma $ its tangent cone is the unique supporting hyperplane of the convex hull of $ \Gamma $. This implies that $ \Gamma $ actually coincides with the boundary of its convex hull and it is a $ \mathcal{C}^{1} $ hypersurface. At this point, in contrast with the varifold's case, we cannot conclude using Allard's regularity theory, since such a theory has not been extended to $(m,h)$ sets\footnote{However, it is a natural question to understand if Allard regularity theorem can be proved in the more general setting of $(m,h)$ sets. A result pointing to a possible positive answer is contained in \cite{MR3823884}, where $ \mathcal{C}^{2, \alpha} $ regularity has been proved for $(m,0)$ sets that are graphs of continuous functions.}. Therefore to conclude the proof we use an idea that we have learned from \cite{Men12}. We apply the barrier principle \cite[7.1]{MR3466806} in combination with a result of Federer \cite[3.1.23]{MR0257325} to gain some further regularity for $ \Gamma $, namely it is a $ \mathcal{C}^{1,1} $ hypersurface. At this point the conclusion can be easily deduced from a direct computation. 

The sharp geometric inequality for $(m,h)$ sets readily implies sufficient conditions (see \ref{area-blow-up set} and \ref{area-blow-up set stationary}) to conclude that the area-blow-up set of certain sequences of varifolds is empty.

{\small\paragraph*{Acknowledgements.} Most of the work in section \ref{section: Area formula} was carried out when the author was a Phd student in the Geometric Measure Theory group led by Prof.\ Ulrich Menne at Max Planck Institute for Gravitational Physics. The author thanks Prof.\ Ulrich Menne for many conversations on the subject of the present paper and to have kindly made available his unpublished lecture notes \cite{Men12}, where some of the key ideas of the present work originate from.}
\section{Preliminaries}\label{section: preliminaries}

As a general rule, the notation and the terminology used without comments agree with \cite[pp.\ 669--676]{MR0257325}. For varifolds our terminology is based on \cite[Appendix C]{MR855173}. The symbols $ \mathbf{U}(a,r) $ and $ \mathbf{B}(a,r) $ denote the open and closed ball with centre $ a $ and radius $ r $ (\cite[2.8.1]{MR0257325}); $ \mathbf{S}^{m} $ is the $ m $ dimensional unit sphere in $ \Real{m+1} $ (\cite[3.2.13]{MR0257325}); $ \Leb{m} $ and $ \Haus{m} $ are the $ m $ dimensional Lebesgue and Hausdorff measure (\cite[2.10.2]{MR0257325}); $ \mathbf{G}(m,k) $ is the Grassmann manifold of all $ k $ dimensional subspaces in $ \Real{m} $ (\cite[1.6.2]{MR0257325}). Given a measure $ \mu $, we denote by $ \bm{\Theta}^{m}(\mu, \cdot) $ the $ m $ dimensional density of $ \mu $ (\cite[2.10.19]{MR0257325}). Moreover, given a function $ f $, we denote by $ \dmn f $, $ \im f $ and $ \nabla f $ the domain, the image and the gradient of $ f $. The closure and the boundary in $ \Real{n} $ of a set $ A $ are denoted by $ \overline{A} $ and $ \partial A $ and, if $ \lambda > 0 $ and $ x \in \Real{n} $ then $ \lambda(A-x)= \{ \lambda(y-x) : y \in A \} $. The symbols $ \Tan(A,a) $ and $ \Nor(A,a) $ denote the tangent and the normal cone of $ A $ at $ a $ (\cite[3.1.21]{MR0257325}). The symbol $ \bullet $ denotes the standard inner product of $\Real{n}$. If $T$ is a linear subspace of $\Real{n}$, then $T_{\natural} : \Real{n} \rightarrow \Real{n}$ is the orthogonal projection onto $T$ and $T^{\perp} = \Real{n} \cap \{ v : v \bullet u =0 \; \textrm{for $u \in T$} \}$. If $X$ and $Y$ are sets and $ Z \subseteq X \times Y $ we define
\begin{equation*}
Z | S = Z \cap \{ (x,y) : x \in S  \} \quad \textrm{for $ S \subseteq X $,}
\end{equation*}
\begin{equation*}
Z(x) = Y \cap \{ y : (x,y)\in Z   \} \quad \textrm{for $ x \in X $.}
\end{equation*}
The maps $ \mathbf{p}, \mathbf{q} : \Real{n} \times \Real{n} \rightarrow \Real{n} $ are 
\begin{equation*}
\mathbf{p}(x,v)= x, \quad \mathbf{q}(x,v) = v.
\end{equation*}
If $A \subseteq \Real{n}$ and $ m \geq 1 $ is an integer, we say that \textit{$A$ is countably $\rect{m}$ rectifiable of class $2$} if $A$ can be $\Haus{m}$ almost covered by the union of countably many $m$ dimensional submanifolds of class $2$ of $\Real{n}$; we omit the prefix ``countably'' when $\Haus{m}(A)<\infty$. If $X$ and $Y$ are metric spaces and $f : X \rightarrow Y$ is a function such that $f$ and $f^{-1}$ are Lipschitzian functions, then we say that $f$ is a \emph{bi-Lipschitzian homeomorphism}.

\subsection*{Approximate second fundamental form}

In this paper we employ weak notions of second fundamental form and mean curvature that can be naturally associated to each set $ A \subseteq \Real{n} $ at those points $ a \in \Real{n} $ where $ A $ is approximately differentiable of order $ 2 $ in the sense of \cite{San}. In order to keep this preliminary section relatively short we directly refer to \cite[2.7-2.11]{Santilli2019}, where relevant definitions and remarks about the theory developed in \cite{San} are summarized. On the basis of \cite[2.7-2.8]{Santilli2019} we can introduce the following definitions.

\begin{Definition}\label{ap second fundamental form}
	The \emph{approximate second fundamental form of $ A $ at $ a $} is
	\begin{equation*}
	\ap \mathbf{b}_{A}(a) =\ap \Der^{2}A(a)| \ap\Tan(A,a)\times \ap \Tan(A,a)
	\end{equation*}
	and the associated \emph{approximate mean curvature of $ A $ at $ a $} is
	\begin{equation*}
	\ap \mathbf{h}_{A}(a) = \trace \big(\ap \mathbf{b}_{A}(a)\big).
	\end{equation*}
\end{Definition}

If $ A $ is an $ m $ dimensional submanifold of class $ 2 $ then these notions agree with the classical notions from differential geometry, see \cite[2.9]{Santilli2019}.

\subsection*{Curvature for arbitrary closed sets} 

Besides the concept of approximate second fundamental form, in this paper we make use of a more general notion of second fundamental form introduced in \cite{Santilli2019} that can be associated to arbitrary closed sets. The theory of curvature for arbitrary closed sets has been developed in \cite{MR534512}, \cite{MR2031455}, \cite{Santilli2019} and here we summarize those concepts that are relevant for our purpose in the present paper.

Suppose $A$ is a closed subset of $ \Real{n} $.

\begin{Miniremark}\label{curvature of arbitrary closed sets 1}
(cf.\ \cite[2.12, 3.1]{Santilli2019}) The \emph{distance function to $ A $} is denoted by $\bm{\delta}_{A} $ and $S(A,r) = \{ x : \bm{\delta}_{A}(x) = r  \}$. It follows from \cite[2.13]{Santilli2019} that if $ r > 0 $ then $ \Haus{n-1}(S(A,r)\cap K)< \infty $ whenever $ K \subseteq \Real{n} $ is compact and $ S(A,r) $ is countably $ \rect{n-1} $ rectifiable of class $ 2 $. 

If $U$ is the set of all $x \in \Real{n}$ such that there exists a unique $a \in A$ with $|x-a| = \bm{\delta}_{A}(x)$, we define the \textit{nearest point projection onto~$A$} as
		the map $\bm{\xi}_{A}$ characterised by the requirement
		\begin{equation*}
		| x- \bm{\xi}_{A}(x)| = \bm{\delta}_{A}(x) \quad \textrm{for $x \in U$}.
		\end{equation*}
		Let $U(A) = \dmn \bm{\xi}_{A} \sim A$. The functions $ \bm{\nu}_{A} $ and $ \bm{\psi}_{A} $ are defined by
		\begin{equation*}
		\bm{\nu}_{A}(z) = \bm{\delta}_{A}(z)^{-1}(z -  \bm{\xi}_{A}(z)) \quad \textrm{and} \quad \bm{\psi}_{A}(z)= (\bm{\xi}_{A}(z), \bm{\nu}_{A}(z)),
		\end{equation*}
		whenever $ z \in U(A)$.
	\end{Miniremark}
\begin{Miniremark}\label{curvature of arbitrary closed sets 2}
(cf.\ \cite[3.6, 3.13]{Santilli2019}) We define the Borel function $ \rho(A, \cdot) $ setting
\begin{equation*}
\rho(A,x) = \sup \{t : \bm{\delta}_{A}(\bm{\xi}_{A}(x) + t (x-\bm{\xi}_{A}(x) ))=t \bm{\delta}_{A}(x)  \} \quad \textrm{for $ x \in U(A) $,}
\end{equation*}
and we say that $ x \in U(A) $ is a \emph{regular point of $ \bm{\xi}_{A} $} provided that $ \bm{\xi}_{A}$ is approximately differentiable at $ x $ with symmetric approximate differential and $ \ap \lim_{y \to x} \rho(A,y) = \rho(A,x)>1$ (see \cite[2.4, 2.5]{Santilli2019} for the definition of approximate limit and approximate differentiability). The set of regular points of $ \bm{\xi}_{A} $ is denoted by $ R(A)$.
\end{Miniremark}

\begin{Miniremark}\label{curvature of arbitrary closed sets 3}
 (cf.\ \cite[4.1, 4.4, 4.7, 4.9]{Santilli2019}) \emph{The generalized unit normal bundle of $ A $} is defined as
\begin{equation*}
N(A) = (A \times \mathbf{S}^{n-1}) \cap \{ (a,u) :  \bm{\delta}_{A}(a+su)=s \; \textrm{for some $ s > 0 $}\}
\end{equation*}
and $ N(A,a) = \{ v : (a,v) \in N(A)   \} $ for $ a \in A $. 

If $ x \in R(A) $ then we say that \emph{$ \bm{\psi}_{A}(x) $ is a regular point of $ N(A) $.} We denote the set of all regular points of $N(A)$ by $R(N(A))$. For every $ (a,u) \in R(N(A)) $ we define
\begin{equation*}
T_{A}(a,u) = \im \ap \Der \bm{\xi}_{A}(x) \quad \textrm{and} \quad Q_{A}(a,u)(\tau,\tau_{1}) = \tau \bullet \ap \Der \bm{\nu}_{A}(x)(v_{1}),
\end{equation*}
where $ x $ is a regular point of $ \bm{\xi}_{A}$ such that $ \bm{\psi}_{A}(x)= (a,u) $, $ \tau \in T_{A}(a,u) $, $  \tau_{1} \in T_{A}(a,u) $ and $ v_{1} \in \Real{n} $ such that $ \ap \Der \bm{\xi}_{A}(x)(v_{1}) = \tau_{1}$. We say that $ Q_{A}(a,u) $ is the \textit{second fundamental form of $ A $ at $ a $ in the direction $ u $.} 

If $(a,u) \in R(N(A))$ \emph{the principal curvatures of $ A $ at $(a,u) $} are the numbers
\begin{equation*}
\kappa_{A,1}(a,u) \leq \ldots \leq \kappa_{A,n-1}(a,u),
\end{equation*}
defined so that $\kappa_{A,m +1}(a,u) = \infty $, $\kappa_{A,1}(a,u), \ldots , \kappa_{A,m}(a,u)$ are the eigenvalues of $ Q_{A}(a,u)$ and $ m = \dim T_{A}(a,u)$. Moreover
\begin{equation*}
\chi_{A,1}(x) \leq \ldots \leq \chi_{A,n-1}(x)
\end{equation*}
are the eigenvalues of $ \ap \Der \bm{\nu}_{A}(x)| \{v : v \bullet \bm{\nu}_{A}(x)=0  \} $ for $ x \in R(A) $. 

It follows from \cite[4.10]{Santilli2019} that if $ r > 0 $ and $ x \in S(A,r) \cap R(A) $ then
\begin{equation}\label{curvature of arbitrary closed sets 3: eq1}
\kappa_{A,i}(\bm{\psi}_{A}(x)) = \chi_{A,i}(x)(1-r\chi_{A,i}(x))^{-1} \quad \textrm{for $ i = 1, \ldots , n-1 $.}
\end{equation}

\end{Miniremark}
\begin{Miniremark}\label{curvature of arbitrary closed sets 4} 
(cf.\ \cite[5.1, 5.2]{Santilli2019})	For each $ a \in A $ we define the closed convex subset
	\begin{equation*}
	\Dis(A,a) = \{ v : |v| = \bm{\delta}_{A}(a+v)  \}
	\end{equation*}
	and we notice that $ N(A,a) = \{  v/|v|: 0 \neq v \in \Dis(A,a)  \} $. For every integer $ 0 \leq m \leq n $ we define \textit{the $m$-th stratum of $A$} by
	\begin{equation*}
	A^{(m)} = A \cap \{ a : \dim \Dis(A,a) = n-m \};
	\end{equation*}
	this is a Borel set which is countably $ m $ rectifiable and countably $ \rect{m} $ rectifiable of class $ 2 $; see \cite[4.12]{MR4012808}. 
	
	The following assertion will be useful: \emph{if $ a \in A^{(m)} $ then
	\begin{equation*}
	\Haus{n-m-1}(N(A,a) \cap V) > 0
	\end{equation*}	
whenever $ V $ is an open subset of $ \Real{n} $ such that $ V \cap N(A,a) \neq \varnothing $}. In fact, noting that $ \Haus{n-m}(U \cap \Dis(A,a)) > 0 $ whenever $ U $ is open and $ U \cap \Dis(A,a) \neq \varnothing $, the assertion follows applying Coarea formula \cite[3.2.22(3)]{MR0257325}.
\end{Miniremark}

The relation between the two notions of second fundamental form defined in \ref{ap second fundamental form} and \ref{curvature of arbitrary closed sets 3} is given by the following result, proved in \cite[6.2]{Santilli2019}. 

\begin{Theorem}\label{curvature of arbitrary closed sets 5} 
	If $ A \subseteq \Real{n} $ is a closed set, $ 1 \leq m \leq n-1 $  and $ S \subseteq A $ is $ \Haus{m} $ measurable and $ \rect{m} $ rectifiable of class $ 2 $ then there exists $ R \subseteq S $ such that $ \Haus{m}(S \sim R) = 0 $,
	\begin{equation*}
	\textstyle	\ap \Tan(S,a) = T_{A}(a,u)\in \mathbf{G}(n,m), \quad \ap \mathbf{b}_{S}(a)(\tau, \upsilon) \bullet u = - Q_{A}(a,u)(\tau, \upsilon) 
	\end{equation*}
	for every $ \tau, \upsilon \in T_{A}(a,u) $ and for $ \Haus{n-1} $ a.e.\ $ (a,u) \in N(A)|R $.
\end{Theorem}

\begin{Remark}\label{curvature of arbitrary closed sets 5 remark} 
It is in general not possible to replace $ N(A)|R $ with $ N(A)|S $ in the conclusion, even if $ S $ is the boundary of a $ \mathcal{C}^{1,\alpha} $ convex set $ A $; see the example in \cite[6.3]{Santilli2019}.
\end{Remark}

\subsection*{Level sets of the distance function}

We conclude this preliminary section providing a structural result for the level sets of the distance function from an arbitrary closed set, which is sufficient for the purpose of the present work. Other structural results are available, in particular we refer to \cite{MR2954647} and references therein.
  
\begin{Theorem}[Gariepy-Pepe]\label{Gariepy-Pepe}
Suppose $ A $ is a closed subset of $ \Real{n} $, $ r > 0 $, $ x \in S(A,r) $, $ \bm{\delta}_{A} $ is differentiable at $ x $ and $ T = \{ v : v \bullet \nabla \bm{\delta}_{A}(x) =0 \} $. 

Then there exists an open neighborhood $ V $ of $ x $ and a Lipschitzian function $ f : T \rightarrow T^{\perp} $ such that $ f $ is differentiable at $ T_{\natural}(x) $ with $ \Der f(T_{\natural}(x)) =0 $ and
\begin{equation*}
	 V \cap S(A,r) = V \cap \{ \chi + f(\chi) : \chi \in T \}.
\end{equation*}
\end{Theorem}

\begin{proof}
The arguments in the proof of \cite[Theorem 1]{MR0287442} prove the statement with the exception of the differentiability properties of $ f $, which can be easily deduced\footnote{In fact the following statement follows from the definition of tangent cone (see \cite[3.1.21]{MR0257325}). \emph{If $ T \in \mathbf{G}(n, n-1) $, $ \alpha \in T $, $ f : T \rightarrow T^{\perp} $ is continuous at $\alpha$, $ a = \alpha + f(\alpha) $, $ A = \{ \chi + f(\chi): \chi \in T  \} $ and $ \Tan(A,a) \subseteq T $ then $f$ is differentiable at $ \alpha $ with $\Der f(\alpha) =0 $.}} noting that $ \Tan(S(A,r),x) \subseteq T $.
\end{proof}

\begin{Lemma}\label{local structure of level sets}
	If $ A \subseteq \Real{n} $ is a closed set then the following conclusion holds for $ \Leb{1} $ a.e.\ $ r > 0 $ and for $ \Haus{n-1} $ a.e.\ $ x \in S(A,r) $:
	\begin{equation*}
	\Tan(S(A,r),x) = \ap \Tan(S(A,r),x) = \{ v : v \bullet \bm{\nu}_{A}(x)=0  \}
	\end{equation*}
and, if $ T = \Tan(S(A,r),x) $, there exists an open neighborhood $ V $ of $ x $ and a Lipschitzian function $ f : T \rightarrow T^{\perp} $ such that $ f $ is pointwise differentiable of \mbox{order $ 2 $} at $T_{\natural}(x)$, $\Der f(T_{\natural}(x))=0 $, 
	\begin{equation*}
\textstyle \Der^{2}f(T_{\natural}(x))(u,v) \bullet \bm{\nu}_{A}(x) = - \ap \Der \bm{\nu}_{A}(x)(u) \bullet v \quad \textrm{for $ u,v \in T $}
	\end{equation*}
and $	V \cap S(A,r) = V \cap \{ \chi + f(\chi) : \chi \in T \} $.
\end{Lemma}

\begin{proof}
	Since $ \bm{\delta}_{A} $ is differentiable at $ \Leb{n} $ a.e.\ $ x \in \Real{n} $, it follows from \cite[4.8(3)]{MR0110078} and Coarea formula that $ \bm{\nu}_{A}(x) = \nabla \bm{\delta}_{A}(x) $ for $ \Haus{n-1} $ a.e.\ $ x \in S(A,r) $ and for $ \Leb{1} $ a.e.\ $ r > 0 $. Henceforth, it follows from \cite[3.14]{MR3936235} and \ref{Gariepy-Pepe} that for $ \Leb{1} $ a.e.\ $ r > 0 $ the level set $ S(A,r) $ is pointwise differentiable of order $ 1 $ at $ \Haus{n-1} $ a.e.\ $ x \in S(A,r) $ with 
\begin{equation*}
\Tan(S(A,r),x) = \{v : v \bullet \bm{\nu}_{A}(x) =0  \}.
\end{equation*}
Noting \cite[2.16]{Santilli2019}, we can argue as in the first paragraph of \cite[3.12]{Santilli2019} to infer that for all $ \Leb{1} $ a.e.\ $ r > 0 $ and for $ \Haus{n-1} $ a.e.\ $ x \in S(A,r) $ there exists $ s > 0 $ such that
\begin{equation*}
\mathbf{U}(x+s\bm{\nu}_{A}(x),s) \cap S(A,r) = \varnothing; 
\end{equation*}
therefore, since it is obvious that for every $ x \in S(A,r) $ there exists $ a \in A $ such that $|x-a| = r $ and $ \mathbf{U}(a,r)\cap S(A,r) = \varnothing $, it follows that
\begin{equation*}
\limsup_{t \to 0}t^{-2}\sup\{\bm{\delta}_{\Tan(S(A,r),x)}(z-x) : z \in \mathbf{U}(x,t)\cap S(A,r) \}< \infty
\end{equation*}
for $ \Haus{n-1} $ a.e.\ $ x \in S(A,r) $ and for $ \Leb{1} $ a.e.\ $ r > 0 $. It follows that $ S(A,r) $ is pointwise differentiable of order $(1,1)$ at $ \Haus{n-1} $ a.e.\ $ x \in S(A,r) $ and for $ \Leb{1} $ a.e.\ $ r > 0 $ (see \cite[3.3]{MR3936235}) and we employ \cite[5.7(3)]{MR3936235} to conclude that $ S(A,r) $ is pointwise differentiable of order $ 2 $ at $ \Haus{n-1} $ a.e.\ $ x \in S(A,r) $ and for $ \Leb{1} $ a.e.\ $ r > 0 $. Now the conclusion can be easily deduced with the help of \ref{Gariepy-Pepe}, \cite[3.14, footnote of 3.12]{MR3936235} and \cite[3.12]{Santilli2019}.
\end{proof}

\section{Area formula for the spherical image}\label{section: Area formula}

We introduce now the key concept of Lusin (N) condition for the generalized unit normal bundle.

\begin{Definition}\label{Lusin Property}
	Suppose $ A \subseteq \Real{n} $ is a closed set, $ \Omega \subseteq \Real{n} $ is an open set and $ 1 \leq m < n $ is an integer. We say that $ N(A) $ satisfies the \textit{$ m $ dimensional Lusin (N) condition in $ \Omega $} if and only if (see \ref{curvature of arbitrary closed sets 3}-\ref{curvature of arbitrary closed sets 4})
	\begin{equation*}
	\Haus{n-1}(N(A)|S)=0 \quad \textrm{whenever $S \subseteq  A \cap \Omega$ such that $\Haus{m}(A^{(m)} \cap S) =0 $.}
	\end{equation*}
\end{Definition}

\begin{Remark}\label{Lusin condition and dimension of tangent space}
	If $N(A)$ satisfies the $ m $ dimensional Lusin (N) condition in $ \Omega $ then it follows from \cite[6.1]{Santilli2019} and \cite[4.12]{MR4012808} that 
	\begin{equation*}
	\dim T_{A}(a,u) = m \quad \textrm{for $ \Haus{n-1} $ a.e.\ $ (a,u) \in N(A)| \Omega $.}
	\end{equation*}
\end{Remark}

The following coarea-type formula is a crucial consequence of the Lusin (N) condition.
\begin{Theorem}\label{area formula for the gauss map}
	Suppose $ 1 \leq m < n $ is an integer, $ \Omega \subseteq \Real{n} $ is open, $ A \subseteq \Real{n} $ is closed and $N(A)$ satisfies the $ m $ dimensional Lusin (N) condition in $ \Omega $.
	
	Then for every $\Haus{n-1}$ measurable set $ B \subseteq N(A)|\Omega $,
	\begin{equation*}
	 \int_{\mathbf{S}^{n-1}} \Haus{0}\{ a: (a,u)\in B  \}\,d\Haus{n-1}u  = \int_{A}\int_{B(z) } | \discr Q_{A}| d\Haus{n-m-1} d\Haus{m}z.
	\end{equation*}
\end{Theorem}

\begin{proof}
It follows from \ref{Lusin condition and dimension of tangent space} that for $\Haus{n-1}$ a.e.\ $(a,u)\in N(A)| \Omega$,
	\begin{equation*}
	\kappa_{A, m +1}(a,u) = \infty \quad \textrm{and}\quad \discr Q_{A}(a,u) = \prod_{i=1}^{m}\kappa_{A,i}(a,u).
	\end{equation*}
	Therefore we use \cite[4.11(3), 5.4]{Santilli2019} to compute
	\begin{flalign*}
	& \int_{\mathbf{S}^{n-1}} \Haus{0}\{ a: (a,v)\in B  \} d\Haus{n-1}v\\
	& \quad =   \int_{B}  \prod_{i=1}^{n-1}|\kappa_{A,i}(a,u)|(1+\kappa_{A,i}(a,u)^{2})^{-1/2} d\Haus{n-1}(a,u) \\
	& \quad =  \int_{B|A^{(m)}} |\discr Q_{A}(a,u)| \prod_{i=1}^{m}(1+\kappa_{A,i}(a,u)^{2})^{-1/2}  d\Haus{n-1}(a,u) \\
	& \quad = \int_{A^{(m)}}\int_{B(z)} | \discr Q_{A}| d\Haus{n-m-1} d\Haus{m}z,
	\end{flalign*}
	whenever $B \subseteq N(A)|\Omega$ is $\Haus{n-1}$ measurable.
\end{proof}

We point out a simple and very useful generalization of the barrier principle in \cite[7.1]{MR3466806}.

\begin{Lemma}\label{weak maximum principle}
	Suppose $ 1 \leq m < n $ are integers, $ T \in \mathbf{G}(n,n-1) $, $ \eta \in T^{\perp} $, $ f : T \rightarrow T^{\perp} $ is pointwise differentiable of order $ 2 $ at $ 0 $ such that $ f(0) =0 $ and $ \Der f(0) =0 $, $ h \geq 0 $, $ \Omega $ is an open subset of $ \Real{n} $ and $ \Gamma $ is an $ (m,h) $ subset of $ \Omega $ such that $ 0 \in \Gamma $ and 
	\begin{equation*}
	\Gamma \cap V \subseteq \{ z : z \bullet \eta \leq f(T_{\natural}(z)) \bullet \eta  \}
	\end{equation*}
	for some open neighbourhood $ V $ of $ 0 $.  
	Then, denoting by $ \chi_{1} \geq \ldots \geq \chi_{n-1} $ the eigenvalues of $ \Der^{2}f(0)\bullet \eta $, it follows that
	\begin{equation*}
	\chi_{1} + \ldots + \chi_{m} \geq -h.
	\end{equation*}
\end{Lemma}

\begin{proof}
	Fix $ \epsilon > 0 $. We define
	\begin{equation*}
	\psi(\chi) = \Big(\frac{1}{2}\Der^{2}f(0)(\chi,\chi)\bullet \eta + \epsilon |\chi|^{2}\Big) \eta \quad \textrm{for $ \chi \in T $},
	\end{equation*}
	\begin{equation*}
	M = \{\chi + \psi(\chi) : \chi \in T   \},
	\end{equation*}
	and we select $ r > 0 $ such that $ f(\chi) \bullet \eta\leq \psi(\chi) \bullet \eta $ for $ \chi \in \mathbf{U}(0,r) \cap T $. By \cite[7.1]{MR3466806}, if $ \kappa_{1} \leq \ldots \leq \kappa_{n-1} $ are the principal curvatures at $ 0 $ of $ M $ with respect to the unit normal that points into $ \{ z : z \bullet \eta \leq \psi(T_{\natural}(z)) \bullet \eta \} $, then
	\begin{equation*}
	\kappa_{1} + \ldots + \kappa_{m} \leq h.
	\end{equation*}
	Since a standard and straightforward computation shows that $ \kappa_{i} = -\chi_{i} -\epsilon $ for $ i = 1, \ldots , n-1 $, we obtain the conclusion letting $ \epsilon \to 0 $.
\end{proof}

Finally the following immediate consequence of Federer's Coarea formula is needed.

\begin{Lemma}\label{Coarea: consequence}
	Suppose $ 0 \leq \mu \leq m $ are integers, $ W $ is a $ \rect{m} $ rectifiable and $ \Haus{m} $ measurable subset of $ \Real{n} $, $ S \subseteq \Real{\nu} $ is a countable union of sets with finite $ \Haus{\mu}$ measure and $ f: W \rightarrow \Real{\nu} $ is a Lipschitzian map such that 
	\begin{equation*}
	\Haus{m}(W \cap \{ w: \| \textstyle \bigwedge_{\mu}\big((\Haus{m}\restrict W,m)\ap \Der f(w)\big) \| =0  \}) = 0,
	\end{equation*}
	\begin{equation*}
	\Haus{\mu}(S \cap \{ z : \Haus{m-\mu}(f^{-1}\{z\}) > 0 \}) = 0.
	\end{equation*}
	
	Then $ \Haus{m}(f^{-1}[S])= 0 $.
\end{Lemma}

\begin{proof}
	Firstly we reduce the problem to the case $ \Haus{\mu}(S) < \infty $; then, by \cite[2.1.4, 2.10.26]{MR0257325}, to the case of a Borel subset $S$ of $\Real{\nu}$. Now the conclusion comes from the coarea formula in \cite[p.\ 300]{MR0467473}.
\end{proof}

\begin{Remark}\label{Coarea: consequence remark}
	If $ m = \mu $ then the result is true even if we omit to assume that $ S $ is a countable union of sets with finite $ \Haus{\mu}$ measure, as one may check noting that $\{ z : \Haus{0}(f^{-1}\{z\}) > 0 \} = f[W] $ and applying \ref{Coarea: consequence} with $ S $ replaced by $ S \cap f[W] $.
\end{Remark}

In the proof of the next result it is convenient to introduce the following Borel sets (see \cite[3.8]{Santilli2019}).

\begin{Definition}
If $ A \subseteq \Real{n} $ is closed and $ \lambda \geq 1 $ we define (see \ref{curvature of arbitrary closed sets 2})
\begin{equation*}
A_{\lambda} = \{ x : \rho(A,x) \geq \lambda \}.
\end{equation*}
\end{Definition}

We are now in the position to prove the main result of this section.

\begin{Theorem}\label{Lusin property for (m,h) sets}
	Suppose $ 1 \leq m \leq n-1 $, $ \Omega $ is an open subset of $ \Real{n} $, $ 0 \leq h < \infty $, $\Gamma$ is an $ (m,h) $ subset of $ \Omega $ that is a countable union of sets with finite $ \Haus{m} $ measure and $ A = \overline{\Gamma} $. 
	
	Then the following two statements hold:
	\begin{enumerate}
	\item\label{Lusin property for (m,h) sets (1)} $ N(A) $ satisfies the $ m $ dimensional Lusin (N) condition in $ \Omega $;
	\item\label{Lusin property for (m,h) sets (2)} 	for $ \Haus{n-1} $ a.e.\ $ (z, \eta) \in N(A) | \Omega $,
	\begin{equation*}
	\dim T_{A}(z,\eta) = m \quad \textrm{and} \quad	\trace Q_{A}(z, \eta) \leq h.
	\end{equation*}
	\end{enumerate}
	 \end{Theorem}

\begin{proof}
We divide the proof is several claims. Fix $\tau > 2m $.
	
\textbf{Claim 1.} \emph{If $ 0 < r < \frac{m}{3(2m-1)h} $ and $ x \in S(A,r) \cap R(A) \cap A_{\tau} \cap \bm{\xi}_{A}^{-1}(\Gamma) $ (see \ref{curvature of arbitrary closed sets 2}) are such that $ \Hdensity{n-1}{S(A,r) \sim A_{\tau}}{x}=0 $ and the conclusion of \ref{local structure of level sets} holds, then	
	\begin{equation*}
\sum_{i=1}^{m}\chi_{A,i}(x)\leq h \quad \textrm{and} \quad	\| \textstyle \bigwedge_{m}\big((\Haus{n-1}\restrict S(A,r),n-1)\ap \Der \bm{\xi}_{A}(x)\big) \| > 0.
	\end{equation*}}
Noting that $ \bm{\xi}_{A}|A_{2m} $ is approximately differentiable at $ x $, we employ \cite[3.10(3)(6)]{Santilli2019} and \cite[3.2.16]{MR0257325} to conclude that
\begin{equation}\label{Lusin property for (m,h) sets:2}
 \chi_{A,j}(x) \geq -(2m -1)^{-1}r^{-1} \quad \textrm{for $ i = 1, \ldots , n-1 $}
\end{equation}
\begin{equation}\label{Lusin property for (m,h) sets:4}
\ap \Der \bm{\xi}_{A}(x)| \Tan(\Haus{n-1}\restrict S(A,r),x) = (\Haus{n-1} \restrict S(A,r), n-1)\ap \Der \bm{\xi}_{A}(x).
\end{equation}
We assume $ \bm{\xi}_{A}(x) =0 $ and we notice that $ T_{\natural}(x)=0 $ and $	\bm{\nu}_{A}(x) = r^{-1}x$. We choose $ f $, $ V $ and $ T $ as in \ref{local structure of level sets} and $ 0 < s < r/2 $ such that $ \mathbf{U}(x,s) \subseteq V $. Then we define $ g(\zeta) = f(\zeta) - x $ for $ \zeta \in T $, 
	\begin{equation*}
	U= T_{\natural}\big(\mathbf{U}(x,s) \cap \{ \chi + f(\chi) : \chi \in T \}\big), \quad W = \{ y-x: y \in T_{\natural}^{-1}(U) \cap \mathbf{U}(x,s)  \}.
	\end{equation*}
	It follows that $ W $ is an open neighbourhood of $ 0 $ and 
	\begin{equation}\label{Lusin property for (m,h) sets:5}
	W \cap A \subseteq \{ z : z \bullet \bm{\nu}_{A}(x) \leq g(T_{\natural}(z)) \bullet \bm{\nu}_{A}(x) \}.
	\end{equation}
	If \eqref{Lusin property for (m,h) sets:5} did not hold then there would be $y \in \mathbf{U}(x,s) \cap T_{\natural}^{-1}[U]$ such that $y-x \in A $ and $ y \bullet \bm{\nu}_{A}(x) > f(T_{\natural}(y))\bullet \bm{\nu}_{A}(x) $; noting that
	\begin{equation*}
	T_{\natural}(y) + f(T_{\natural}(y)) \in \mathbf{U}(x,s) \cap S(A,r), \quad |T_{\natural}(y) + f(T_{\natural}(y))-y| < r,
	\end{equation*}
	we would conclude
	\begin{equation*}
	|T_{\natural}(y) + f(T_{\natural}(y))-(y-x)| = r - (y-f(T_{\natural}(y)))\bullet \bm{\nu}_{A}(x)< r
	\end{equation*}
	which is a contradiction. Since $ -\chi_{A,1}(x), \ldots , - \chi_{A,n-1}(x) $ are the eigenvalues of $ \Der^{2} g(0) \bullet \bm{\nu}_{A}(x) $, we may apply \ref{weak maximum principle} to infer that 
	\begin{equation}\label{Lusin property for (m,h) sets:3}
	\chi_{A,1}(x) + \ldots + \chi_{A,m}(x) \leq h
	\end{equation}
and combining \eqref{Lusin property for (m,h) sets:2} and \eqref{Lusin property for (m,h) sets:3} it follows that
	\begin{equation*}
	\chi_{A,j}(x) \leq \frac{4m-3}{6m-3}r^{-1}< r^{-1} \quad \textrm{for $ j = 1, \ldots , m $.}
	\end{equation*}
Since it follows by \eqref{Lusin property for (m,h) sets:4} and \cite[3.5]{Santilli2019} that $ 1-r\chi_{A,i}(x) $ are the eigenvalues of $(\Haus{n-1} \restrict S(A,r), n-1)\ap \Der \bm{\xi}_{A}(x)$ for $ i = 1, \ldots , n-1 $, we get that
\begin{equation*}
\| {\textstyle \bigwedge_{m}}\big((\Haus{n-1}\restrict S(A,r),n-1)\ap \Der \bm{\xi}_{A}(x)\big) \| \geq \prod_{i=1}^{m}\big(1-\chi_{A,i}(x)r\big) > 0.
\end{equation*}

\textbf{Claim 2.} \emph{For $ \Haus{n-1} $ a.e.\ $ x \in S(A,r) \cap A_{\tau} \cap \bm{\xi}_{A}^{-1}(\Gamma) $ and for $ \Leb{1} $ a.e.\ $ 0 < r < \frac{m}{3(2m-1)h} $ the following inequalities hold:
\begin{equation*}
\sum_{i=1}^{m}\chi_{A,i}(x)\leq h \quad \textrm{and} \quad	\| \textstyle \bigwedge_{m}\big((\Haus{n-1}\restrict S(A,r),n-1)\ap \Der \bm{\xi}_{A}(x)\big) \| > 0.
\end{equation*}}	
Notice that
\begin{equation*}
	\Hdensity{n-1}{S(A,r)\sim A_{\tau}}{x} =0 
\end{equation*} 
for $ \Haus{n-1} $ a.e.\ $ x \in S(A,r) \cap A_{\tau} $ and for every $ r > 0 $ by \cite[2.13(1)]{Santilli2019} and \cite[2.10.19(4)]{MR0257325}, and $ \Haus{n-1}(S(A,r)\sim R(A)) =0 $ for $ \Leb{1} $ a.e.\ $ r > 0 $ by \cite[3.15]{Santilli2019}. Then Claim 2 follows from \ref{local structure of level sets} and Claim 1.

\textbf{Claim 3.} \emph{$N(A)$ satisfies the $ m $ dimensional Lusin (N) condition in $ \Omega $.}\\ Let $ S \subseteq \Gamma $ such that $ \Haus{m}(S \cap A^{(m)}) =0 $. For $ r > 0 $ it follows from \cite[3.16, 3.17(1), 4.3]{Santilli2019} that $ \bm{\psi}_{A}|A_{\tau} \cap S(A,r) $ is a bi-Lipschitzian homeomorphism and 
\begin{equation*}
\bm{\psi}_{A}(\bm{\xi}_{A}^{-1}(x) \cap A_{\tau} \cap S(A,r)  ) \subseteq N(A,x) \quad \textrm{for $ x \in A $};
\end{equation*}
then we apply \cite[5.2]{Santilli2019} to get
	\begin{equation*}
	A \cap	\{ x : \Haus{n-m-1}(\bm{\xi}_{A}^{-1}\{x\} \cap A_{\tau} \cap S(A,r) ) >0  \} \subseteq \bigcup_{i=0}^{m}A^{(i)} \quad \textrm{for every $ r > 0 $.}
	\end{equation*}
Since $ \Haus{m}(A^{(i)}) =0 $ for $ i = 0, \ldots , m-1 $ (see \ref{curvature of arbitrary closed sets 4}), it follows
	\begin{equation*}
	\Haus{m}\big( S \cap \{ x : \Haus{n-m-1}(\bm{\xi}_{A}^{-1}\{x\} \cap A_{\tau} \cap S(A,r) ) >0  \} \big) =0 \quad \textrm{for every $ r > 0 $.}
	\end{equation*}
	Noting Claim 2 and \cite[3.10(1)]{Santilli2019}, we can apply \ref{Coarea: consequence} with $ W $ and $ f $ replaced by $ S(A,r) \cap A_{\tau} \cap \bm{\xi}_{A}^{-1}(\Gamma) $ and $ \bm{\xi}_{A}|S(A,r) \cap A_{\tau} \cap \bm{\xi}_{A}^{-1}(\Gamma) $ to infer that
	\begin{equation*}
	\Haus{n-1}(\bm{\xi}_{A}^{-1}(S) \cap S(A,r) \cap A_{\tau}  ) =0 \quad \textrm{for $ \Leb{1} $ a.e.\  $ 0 < r < \frac{m}{3(2m-1)}h^{-1} $.}
	\end{equation*}
	We notice that $ N(A)|S = \bigcup_{r>0}\bm{\psi}_{A}(S(A,r)\cap A_{\tau} \cap \bm{\xi}_{A}^{-1}(S)) $ by \cite[4.3]{Santilli2019} and $ \bm{\psi}_{A}(S(A,r)\cap A_{\tau}) \subseteq \bm{\psi}_{A}(S(A,s) \cap A_{\tau}) $ if $ s < r $ by \cite[3.17(2)]{Santilli2019}. Henceforth, it follows that
	\begin{equation*}
		\Haus{n-1}(N(A)|S)=0.
	\end{equation*} 
	
	\textbf{Claim 4.} \emph{For $ \Haus{n-1} $ a.e.\ $ (z, \eta) \in N(A) | \Omega $, 
	\begin{equation*}
	\dim T_{A}(z,\eta) = m \quad \textrm{and} \quad \trace Q_{A}(z, \eta) \leq h.
	\end{equation*}}
		By Claim 3, \ref{Lusin condition and dimension of tangent space}, Claim 2 and \eqref{curvature of arbitrary closed sets 3: eq1} it follows that
		\begin{equation*}
			\dim T_{A}(z,\eta) = m \quad \textrm{for $ \Haus{n-1} $ a.e.\ $ (z, \eta) \in N(A)|\Omega $,}
		\end{equation*}
	\begin{equation}\label{Lusin property for (m,h) sets:1}
	\sum_{l=1}^{m}\frac{\kappa_{A,l}(\bm{\psi}_{A}(x))}{1 + r \kappa_{A,l}(\bm{\psi}_{A}(x))} \leq h 
	\end{equation}
	for $ \Haus{n-1} $ a.e.\ $ x \in S(A,r) \cap A_{\tau} \cap \bm{\xi}_{A}^{-1}(\Gamma) $ and for $ \Leb{1} $ a.e.\ $ 0 < r < \frac{m}{3(2m-1)}h^{-1} $. We choose a positive sequence $ r_{i}  \to 0 $ such that if $ M_{i} $ is the set of of points $ x \in S(A, r_{i}) \cap A_{\tau} \cap \bm{\xi}_{A}^{-1}(\Gamma) $ satisfying \eqref{Lusin property for (m,h) sets:1} with $ r $ replaced by $ r_{i} $, then
	\begin{equation*}
	\Haus{n-1}\big(\big(S(A,r_{i})\cap A_{\tau} \cap \bm{\xi}_{A}^{-1}(\Gamma)\big) \sim M_{i}\big) =0  \quad \textrm{for every $ i \geq 1 $.}
	\end{equation*}
	It follows that
	\begin{equation*}
	\trace Q_{A}(z,\eta) \leq h \quad \textrm{if $(z,\eta) \in \bigcap_{i=1}^{\infty}\bigcup_{j=i}^{\infty}\bm{\psi}_{A}(M_{j})$}
	\end{equation*}
and the inclusion
	\begin{equation*}
	(N(A)|\Omega) \sim \bigcap_{i=1}^{\infty}\bigcup_{j=i}^{\infty}\bm{\psi}_{A}(M_{j}) \subseteq \bigcup_{i=1}^{\infty} \Big(\bm{\psi}_{A}\big(S(A,r_{i})\cap A_{\tau} \cap \bm{\xi}_{A}^{-1}(\Gamma) \big) \sim \bm{\psi}_{A}(M_{i})\Big)
	\end{equation*}
	readily implies
	\begin{equation*}
	\Haus{n-1}\Big( (N(A)|\Omega) \sim \bigcap_{i=1}^{\infty}\bigcup_{j=i}^{\infty}\bm{\psi}_{A}(M_{j}) \Big) =0.
	\end{equation*}
\end{proof}

\begin{Remark}
	We assume in \ref{Lusin property for (m,h) sets} that $ \Gamma $ is a countable union of sets of finite $ \Haus{m} $ measure only because this hypothesis ensures the applicability of \ref{Coarea: consequence} in the proof of Claim 3. Consequently, in view of \ref{Coarea: consequence remark}, we have that if $ m = n-1 $ the result is still true even if we omit the aforementioned hypothesis.
\end{Remark}

\begin{Corollary}\label{relation with ap s.f.f.}
	Suppose $ 1 \leq m \leq n-1 $, $ \Omega $ is an open subset of $ \Real{n} $, $ 0 \leq h < \infty $, $\Gamma$ is an $ (m,h) $ subset of $ \Omega $ such that $ \Haus{m}(\Gamma \cap K) < \infty $ for every compact set $ K \subseteq \Omega $, $ A = \overline{\Gamma} $ and $ S = A^{(m)} \cap \Omega $. Then 
	\begin{equation*}
	\ap \Tan(S, z) = T_{A}(z, \eta)\in \mathbf{G}(n,m), \quad \ap \mathbf{b}_{S}(z) \bullet \eta = - Q_{A}(z, \eta)
	\end{equation*}
	for $ \Haus{n-1} $ a.e. $(z, \eta) \in N(A)| \Omega $.
\end{Corollary}

\begin{proof}
If $ K \subseteq \Omega $ is compact then $ S \cap K $ is $ \rect{m}$ rectifiable of class $ 2 $ by \cite[4.12]{MR4012808}. Henceforth the conclusion follows from \ref{Lusin property for (m,h) sets}\eqref{Lusin property for (m,h) sets (1)} and \ref{curvature of arbitrary closed sets 5}.
\end{proof}

\begin{Remark}
Suppose $ V \in \mathbf{V}_{m}(\Omega) $ is an \emph{integral} varifold such that $ \| \delta V \| \leq c \|V\| $ for some $ 0 \leq c < \infty $, $ A = \overline{\spt \|V\|} $ and $ S = A^{(m)} \cap \Omega $. Since $ S \cap K $ is $ \rect{m} $ rectifiable of class $ 2 $ by \cite[4.12]{MR4012808} whenever $ K \subseteq \Omega $ is compact, we use the locality theorem \cite[4.2]{MR2472179} to conclude that
	\begin{equation*}
 \ap \mathbf{h}_{S}(z) =  - \mathbf{h}(V,z)
\end{equation*}
for $ \Haus{m} $ a.e.\ $ z \in S  $. It follows from \ref{Lusin property for (m,h) sets}\eqref{Lusin property for (m,h) sets (1)} and \ref{relation with ap s.f.f.} that
\begin{equation*}
	\trace Q_{A}(z, \eta) = \mathbf{h}(V,z) \bullet \eta \quad \textrm{for $ \Haus{n-1} $ a.e.\ $ (z,\eta) \in N(A)|\Omega $.}
\end{equation*}
Here we consider only \emph{integral} varifolds because the locality theorem \cite[4.2]{MR2472179} is not currently available for non integral ones.
\end{Remark}

\begin{Remark}
	As pointed out in \ref{curvature of arbitrary closed sets 5 remark} the second fundamental form $Q_{\Gamma}$ of an arbitrary closed set $ \Gamma $ when restricted over $ \Gamma^{(m)} $ may not be fully described by $ \ap \mathbf{b}_{\Gamma^{(m)}} $. In a certain sense \ref{relation with ap s.f.f.} draws an interesting analogy with the theory of functions of bounded variation. In fact, it is well known that the total differential of a $BV$ function is not equal to the approximate gradient, unless the function belongs to the Sobolev space. Following this analogy, $(m,h)$ sets correspond to Sobolev functions.
\end{Remark}

\section{Almgren's sharp geometric inequality}

The following lemma will be useful in the proof of the rigidity theorem.

\begin{Lemma}\label{sufficient conditions for bilipschitzian homeomorphisms}
Let $ 1 \leq m \leq n $ be integers and let $ B $ be an $ m $ dimensional submanifold of class $ 1 $ in $ \Real{n} $. If $ 0 < \lambda < 1 $ and $ \varphi : B \rightarrow \Real{n} $ is a Lipschitzian map such that 
\begin{equation*}
	\| \Der(\varphi - \mathbf{1}_{B})(b)\| \leq \lambda \quad \textrm{for $ \Haus{m} $ a.e.\ $ b \in B $},
\end{equation*}
then for each $ b \in B $ there exists an open neighbourhood $ V $ of $ b $ such that $ \varphi|V \cap B $ is a bi-Lipschitzian homeomorphism.
\end{Lemma}

\begin{proof}
First we prove the following claim. \emph{If $ U $ is an open convex subset of $ \Real{m} $, $ 0 \leq M < \infty $ and $ g : U \rightarrow \Real{n} $ is a Lipschitzian map such that $ \|\Der g(x) \| \leq M $ for $ \Leb{m} $ a.e.\ $ x \in U $, then $ \Lip g \leq M $.} In fact, if $ a \in U $ and $ r > 0 $ such that $ \mathbf{U}(a,r) \subseteq U $ then Coarea formula \cite[3.2.22(3)]{MR0257325} and the fundamental theorem of calculus \cite[2.9.20(1)]{MR0257325} imply that for $ \Haus{n-1} $ a.e.\ $ v \in \mathbf{S}^{n-1} $,
\begin{equation*}
	|g(a + tv)-g(a)| \leq M t \quad \textrm{for $ 0 < t < r $};
\end{equation*}
since $ g $ is continuous, 
\begin{equation*}
\limsup_{x \to a}\frac{|g(x)-g(a)|}{|x-a|} \leq M \quad \textrm{for $ a \in U $}
\end{equation*}
and $ \Lip g \leq M $ by \cite[2.2.7]{MR0257325}.

Now we fix $ b \in B $ and $ \sqrt{\lambda} < t < 1$, we define $ A = \Tan(B,b) + b $ and we select $ r>0 $ and diffeomorphism $ f : \Real{n} \rightarrow \Real{n} $ of class $ 1 $ as in \cite[3.1.23]{MR0257325}. In particular we have that $ f(\mathbf{U}(b,tr)) \cap A = f(\mathbf{U}(b,tr) \cap B) $ and
\begin{equation*}
\| \Der(	(\varphi - \mathbf{1}_{B}) \circ f^{-1})(x)\| \leq \lambda t^{-1} \quad \textrm{for $ \Haus{m} $ a.e.\ $ x \in f(\mathbf{U}(b, tr) \cap B) $.}
\end{equation*}
It follows that if $ U $ is a convex subset of $f(\mathbf{U}(b, tr) \cap B) $ such that $ f(b)\in U $ and $ U $ is relatively open in $ A $ then
\begin{equation}\label{sufficient conditions for bilipschitzian homeomorphisms eq1}
\Lip [((\varphi - \mathbf{1}_{B}) \circ f^{-1})|U] \leq \lambda t^{-1}.
\end{equation}
Therefore one uses \eqref{sufficient conditions for bilipschitzian homeomorphisms eq1} and $ \Lip f \leq t^{-1} $ to conclude 
\begin{equation*}
	|\varphi(c) - \varphi(d)| \geq |c-d| (1 - \lambda t^{-2}) \quad \textrm{for $ c,d \in f^{-1}(U) $.}
\end{equation*}
\end{proof}

\begin{Theorem}\label{area mean curvature ch}
If $ 1 \leq m \leq n-1 $, $ h > 0 $ and $ \Gamma $ is a non-empty compact $(m,h)$ subset of $ \Real{n} $ then 
\begin{equation*}
	\Haus{m}(\Gamma) \geq \Big(\frac{m}{h}\Big)^{m} \Haus{m}(\mathbf{S}^{m}).
\end{equation*}

Moreover if the equality holds and $ \Gamma = \spt (\Haus{m}\restrict \Gamma) $ then there exists an $ m+1 $ dimensional plane $ T $ and $ a \in \Real{n} $ such that
\begin{equation*}
	\Gamma = \partial \mathbf{B}(a, m/h) \cap T.
\end{equation*}
\end{Theorem}

\begin{proof}
We assume $ \Haus{m}(\Gamma) < \infty $. Since $ \lambda \Gamma $ is an $(m, h/\lambda) $ set whenever $ \lambda > 0 $, we reduce the proof to the case $ h = m $. 

We define
\begin{equation*}
	C = (\Gamma \times \Real{n}) \cap \{ (a, \eta) : \eta \bullet (w-a) \geq 0 \; \textrm{for every $ w \in \Gamma $}  \}
\end{equation*}
and we notice that $ C $ is a closed subset of $\Gamma \times \Real{n}$, $ C(a) $ is a closed convex cone\footnote{A subset $ C $ of $ \Real{n} $ is a cone if and only if $ \lambda c \in C $ whenever $ 0 < \lambda < \infty $ and $ c \in C $.} containing $ 0 $ for every $ a \in \Gamma $ and, since $ \Gamma $ is compact, for every $ \eta \in \mathbf{S}^{n-1} $ there exists $ a \in \Gamma $ such that $ \inf_{w \in \Gamma}(w \bullet \eta) = ( a \bullet \eta) $; in other words, 
\begin{equation*}
	\mathbf{q}(C \cap (\Gamma \times \mathbf{S}^{n-1})) = \mathbf{S}^{n-1}.
\end{equation*}
Moreover we let $ B = \{ (a,-\eta) : (a,\eta)\in C, \; |\eta| = 1  \} $ and we notice that 
\begin{equation}\label{Area mean curvature characterization: eq 6}
	B \subseteq N(\Gamma) \quad \textrm{and} \quad \mathbf{q}[B] =\mathbf{S}^{n-1}.
\end{equation}

We define $ X \subseteq \Gamma^{(m)} $ as the set of $ a \in \Gamma^{(m)} $ such that the following conditions are satisfied:
\begin{itemize}
\item[(i)] $ \Gamma^{(m)} $ is approximately differentiable of order $ 2 $ at $ a $ with 
\begin{equation*}
\ap\Tan(\Gamma^{(m)},a)\in \mathbf{G}(n,m),
\end{equation*}
\item[(ii)] $ \ap \mathbf{b}_{\Gamma^{(m)}}(a)\bullet \eta = -Q_{\Gamma}(a,\eta) $ for $ \Haus{n-m-1} $ a.e.\ $ \eta \in N(\Gamma,a) $, 
\item[(iii)] $ \trace Q_{\Gamma}(a,\eta) \leq m $ for $ \Haus{n-m-1} $ a.e.\ $ \eta \in N(\Gamma,a) $.
\end{itemize}
Since $ \Gamma^{(m)} $ is $ \rect{m} $ rectifiable of class $ 2 $ by \cite[4.12]{MR4012808}, it follows from \cite[3.23]{San} that condition (i) is satisfied $ \Haus{m} $ almost everywhere on $ \Gamma^{(m)} $. Moreover, noting \ref{Lusin property for (m,h) sets} and \ref{relation with ap s.f.f.}, we infer that
\begin{equation}\label{Area mean curvature characterization: eq 13}
\ap \mathbf{b}_{\Gamma^{(m)}}(a)\bullet \eta = -Q_{\Gamma}(a,\eta) \quad \textrm{and} \quad \trace Q_{\Gamma}(a,\eta) \leq m
\end{equation}
for $ \Haus{n-1} $ a.e.\ $ (a,\eta) \in N(\Gamma) $ and we apply \cite[2.10.25]{MR0257325}, with $ X $ and $ f $ replaced by $ N(\Gamma)| \Gamma^{(m)} $ and $\mathbf{p}$, to conclude that \eqref{Area mean curvature characterization: eq 13} holds for $ \Haus{m} $ a.e.\ $ a \in \Gamma^{(m)} $ and for $ \Haus{n-m-1} $ a.e.\ $ \eta \in N(\Gamma,a) $. Henceforth, 
\begin{equation}\label{Area mean curvature characterization: eq 8}
	\Haus{m}(\Gamma^{(m)} \sim X) =0.
\end{equation}
Furthermore we notice that if $ a \in X $ then
\begin{equation*}
\ap \mathbf{h}_{\Gamma^{(m)}}(a)\bullet \eta \geq -m \quad \textrm{for $ \Haus{n-m-1} $ a.e.\ $ \eta \in N(\Gamma,a) $,}
\end{equation*}
whence we readily infer from \ref{curvature of arbitrary closed sets 4} that
\begin{equation}\label{Area mean curvature characterization: eq 7}
\ap \mathbf{h}_{\Gamma^{(m)}}(a)\bullet \eta \geq -m \quad \textrm{for all $ \eta \in N(\Gamma,a) $.}
\end{equation}

If $ a \in X $ we define 
\begin{equation*}
g(a) = \bm{\xi}_{C(a)}(\ap\mathbf{h}_{\Gamma^{(m)}}(a))
\end{equation*}
(notice $C(a)\subseteq \ap\Nor(\Gamma^{(m)},a) \in \mathbf{G}(n, n-m) $ and $ \ap\mathbf{h}_{\Gamma^{(m)}}(a) \in  \ap\Nor(\Gamma^{(m)},a) $), we infer from \cite[3.9(3)]{MR4012808} that $	\ap\mathbf{h}_{\Gamma^{(m)}}(a)-g(a) \in \Nor(C(a), g(a)) $ and
\begin{equation*}
(\ap\mathbf{h}_{\Gamma^{(m)}}(a)-g(a)) \bullet (\eta-g(a)) \leq 0 \quad \textrm{for every $ \eta \in C(a) $}
\end{equation*}
and, noting that $ 0 \in C(a) $ and $ 2g(a) \in C(a) $, we conclude that
\begin{equation*}
(\ap\mathbf{h}_{\Gamma^{(m)}}(a)-g(a)) \bullet g(a) =0,
\end{equation*}
\begin{equation*}
(\ap\mathbf{h}_{\Gamma^{(m)}}(a)-g(a)) \bullet \eta \leq 0 \quad \textrm{for every $ \eta \in C(a) $.}
\end{equation*}
Since $ -g(a) /|g(a)| \in N(\Gamma, a) $ when $ g(a) \neq 0 $ by \eqref{Area mean curvature characterization: eq 6}, we obtain from \eqref{Area mean curvature characterization: eq 7} that
\begin{equation}\label{Area mean curvature characterization: eq 2}
|g(a)| = \ap\mathbf{h}_{\Gamma^{(m)}}(a) \bullet (g(a)/|g(a)|) \leq m.
\end{equation}
Moreover it follows from \cite[4.12(3)]{San},
\begin{equation*}
\ap\mathbf{b}_{\Gamma^{(m)}}(a)\bullet \eta \geq 0 \quad \textrm{for $ a \in X $ and $ \eta \in C(a) $},
\end{equation*}
whence we deduce 
\begin{equation}\label{Area mean curvature characterization: eq 9}
	g(a)\bullet \eta \geq \ap \mathbf{h}_{\Gamma^{(m)}}(a) \bullet \eta \geq 0\quad \textrm{for $ a \in X $ and $ \eta \in C(a) $},
\end{equation}
and, employing the classical inequality relating the arithmetic and geometric means of a family of non negative numbers\footnote{If $ a_{1}, \ldots , a_{m} $ are non negative real numbers,
	\begin{equation*}
	a_{1}a_{2}\ldots a_{m} \leq \Big(\frac{a_{1}+a_{2}+\ldots +a_{m}}{m}\Big)^{m}
	\end{equation*}
	with equality only if $ a_{1}=a_{2}= \ldots = a_{m} $.},
\begin{equation}\label{Area mean curvature characterization: eq 5}
0 \leq \discr (\ap\mathbf{b}_{\Gamma^{(m)}}(a)\bullet \eta) \leq m^{-m} (\ap\mathbf{h}_{\Gamma^{(m)}}(a)\bullet \eta)^{m} \leq  m^{-m}(g(a)\bullet \eta)^{m}
\end{equation}
for every $ a \in X $ and $ \eta \in C(a) $.

If $ T \in \mathbf{G}(n,m) $ and $ v \in T^{\perp} $ we define
\begin{equation*}
	D(T,v) = T^{\perp} \cap \{ u : u \bullet v \geq 0   \}.
\end{equation*}
We readily infer that there exists $ 0 < \gamma(n,m) < \infty $ such that 
\begin{equation*}
	\gamma(n,m)|v|^{m} = \int_{D(T,v)\cap \mathbf{S}^{n-1}}(\eta \bullet v)^{m}d\Haus{n-m-1}\eta
\end{equation*}
for every $ T \in \mathbf{G}(n,m) $ and $ v \in T^{\perp} $. It follows from \eqref{Area mean curvature characterization: eq 9},
\begin{equation}\label{Area mean curvature characterization: eq 10}
C(a) \subseteq D(\ap\Tan(\Gamma^{(m)},a),g(a)) \quad \textrm{for every $ a \in X $};
\end{equation}
noting \eqref{Area mean curvature characterization: eq 6}, \eqref{Area mean curvature characterization: eq 8}, \eqref{Area mean curvature characterization: eq 5} and \eqref{Area mean curvature characterization: eq 2}, we apply Coarea formula \ref{area formula for the gauss map} to estimate
\begin{flalign*}
& \Haus{n-1}(\mathbf{S}^{n-1}) \\
&\quad \overset{(I)}{\leq}\int_{\mathbf{S}^{n-1}} \Haus{0}\{ a : (a,\eta) \in B  \} d\Haus{n-1}\eta \\
&\quad = \int_{\Gamma}\int_{B(a)}|\discr Q_{\Gamma}(a,\eta)| d\Haus{n-m-1}\eta\, d\Haus{m}a \\
& \quad = \int_{\Gamma^{(m)}}\int_{C(a)\cap \mathbf{S}^{n-1}}\discr (\ap\mathbf{b}_{\Gamma^{(m)}}(a)\bullet \eta)d\Haus{n-m-1}\eta\, d\Haus{m}a \\
&\quad  \overset{(II)}{\leq} m^{-m}\int_{\Gamma^{(m)}}\int_{C(a)\cap \mathbf{S}^{n-1}}(g(a)\bullet \eta)^{m}d\Haus{n-m-1}\eta\, d\Haus{m}a\\
&\quad  \overset{(III)}{\leq} m^{-m}\int_{\Gamma^{(m)}}\int_{D(\ap\Tan(\Gamma^{(m)},a),g(a))\cap \mathbf{S}^{n-1}}(g(a)\bullet \eta)^{m}d\Haus{n-m-1}\eta\, d\Haus{m}a \\
& \quad = m^{-m}\gamma(n,m)\int_{\Gamma^{(m)}}|g(a)|^{m}d\Haus{m}a \\
& \quad \overset{(IV)}{\leq} \gamma(n,m)\Haus{m}(\Gamma^{(m)}) \\
&  \quad \overset{(V)}{\leq} \gamma(n,m)\Haus{m}(\Gamma).
\end{flalign*}

Suppose $ T \in \mathbf{G}(n,m+1) $ and $ \Sigma = T \cap \mathbf{S}^{n-1} $. We observe that if $ a \in \Sigma $ then 
\begin{equation*}
	D(\Tan(\Sigma,a),-a) = \{ \eta : \eta \bullet (w-a) \geq 0 \; \textrm{for every $ w \in \Sigma $}  \},
\end{equation*}
\begin{equation*}
\mathbf{b}_{\Sigma}(a)(u,v)= -a(u \bullet v) \quad \textrm{for $ u,v \in \Tan(\Sigma,a) $,} \quad \mathbf{h}_{\Sigma}(a) = -ma;
\end{equation*}
moreover $ \Sigma = \Sigma^{(m)} $ and 
 \begin{equation*}
 \Haus{0}\{ a: \eta \in D(\Tan(\Sigma,a), -a)    \} = 1 \quad \textrm{for every $ \eta \in \mathbf{S}^{n-1} \sim T^{\perp} $.}
 \end{equation*}
Then we infer that inequalities (I)-(V) in the previous estimate are actually equalities when $ \Gamma = \Sigma $, and we conclude that
\begin{equation*}
\Haus{n-1}(\mathbf{S}^{n-1}) = \gamma(n,m)\Haus{m}(\Sigma) = \gamma(n,m)\Haus{m}(\mathbf{S}^{m}).
\end{equation*}
From this equation we finally obtain that
\begin{equation*}
\Haus{m}(\Gamma) \geq \Haus{m}(\mathbf{S}^{m})
\end{equation*}
and the proof of the first part of the theorem is concluded.

We now assume $ \Haus{m}(\Gamma)= \Haus{m}(\mathbf{S}^{m}) $ and $ \spt (\Haus{m}\restrict \Gamma) = \Gamma $ and we prove that $ \Gamma = f(\mathbf{S}^{m}) $ for some isometric injection $ f : \Real{m+1} \to \Real{n} $.

Firstly, noting that inequalities (I)-(V) are equalities, we infer that
\begin{equation*}
	\Haus{m}(\Gamma \sim \Gamma^{(m)}) =0 \quad \textrm{(by (V))}
\end{equation*}
and the following equalities hold for $ \Haus{m} $ a.e.\ $ a \in \Gamma $,
\begin{equation}\label{Area mean curvature characterization: eq 11}
|g(a)| = m  \quad \textrm{(by (IV) and \eqref{Area mean curvature characterization: eq 2})}, \quad \dim \ap \Tan(\Gamma,a) = m, 
\end{equation}
\begin{equation}\label{Area mean curvature characterization: eq 12}
D(\ap\Tan(\Gamma,a),g(a)) \cap \mathbf{S}^{n-1}=C(a) \cap \mathbf{S}^{n-1} \quad \textrm{(by (III) and \eqref{Area mean curvature characterization: eq 10})},
\end{equation}
\begin{equation}\label{Area mean curvature characterization: eq 3}
 g(a) = \ap \mathbf{h}_{\Gamma}(a), \quad |\ap \mathbf{h}_{\Gamma}(a)| = m, 
\end{equation}
\begin{equation*}
\discr (\ap\mathbf{b}_{\Gamma}(a)\bullet \eta) = m^{-m}(\ap \mathbf{h}_{\Gamma}(a) \bullet \eta)^{m} \quad \textrm{for $ \eta \in C(a) $} \quad \textrm{(by (II) and \eqref{Area mean curvature characterization: eq 5})},
\end{equation*}
\begin{equation}\label{Area mean curvature characterization: eq 4}
	\ap\mathbf{b}_{\Gamma}(a)(u,v) = m^{-1}(u \bullet v) \ap \mathbf{h}_{\Gamma}(a)\quad \textrm{for $ u,v \in \ap \Tan(\Gamma,a) $.}
\end{equation}

Let $ A $ be the convex hull of $ \Gamma $ and let $ B $ be the relative boundary of $ A $. Note that $ A -a \subseteq \{v : v \bullet \eta \geq 0\} $ for every $ a \in \Gamma $ and $ \eta \in C(a) $. Then it follows from \eqref{Area mean curvature characterization: eq 11} and \eqref{Area mean curvature characterization: eq 12} that
\begin{equation*}
\dim \{ u + \lambda g(a): u \in \ap\Tan(\Gamma,a), \; \lambda \geq 0   \} = m+1,
\end{equation*}
\begin{equation*}
A -a \subseteq \{ u + \lambda g(a): u \in \ap\Tan(\Gamma,a), \; \lambda \geq 0   \}
\end{equation*}
for $ \Haus{m} $ a.e.\ $ a \in \Gamma $, whence we deduce that $\dim A \leq m+1$. If $ \dim A = m $ then we could apply \cite[2.9]{Santilli2019} with $ M $ replaced by the relative interior of $ A $ to infer that $ \ap \mathbf{h}_{\Gamma}(x) =0 $ for $ \Haus{m} $ a..e\ $ x \in \Gamma $. Since this contradicts \eqref{Area mean curvature characterization: eq 3} we have proved that $ \dim A = m+1 $. Then we notice that $ \Haus{m}(\Gamma \sim B) = 0 $ and, since $ \Gamma = \spt (\Haus{m}\restrict \Gamma) $, it follows that $ \Gamma \subseteq B $.

At this point it is not restrictive to assume $ m = n-1 $ in the sequel. 

Now we prove that \emph{if $ x \in \Gamma $ then $ \Tan(\Gamma,x) $ is the unique supporting hyperplane of $ A $ at $ x $.} We fix $ x \in \Gamma $. By \cite[1.3.2]{MR3155183} there exists a closed halfspace $ H $ of $ \Real{m+1} $ such that $ 0 \in \partial H $ and $ A -x \subseteq H $. By \cite[Theorem 1.1.7]{MR2458436} we choose a sequence $ \lambda_{i} $ converging to $ +\infty $ and a closed set $ Z $ in $ \Real{m+1} $ such that (see \cite[1.1.1]{MR2458436})
\begin{equation*}
\lambda_{i}(\Gamma -x) \to Z \quad \textrm{as $ i \to \infty $ in the sense of Kuratowski.}
\end{equation*}
Then we notice that $ 0 \in  Z \subseteq H $, $ Z $ is an $ (m,0) $ subset of $ \Real{m+1} $ by \cite[1.6, 3.2]{MR3466806} and $ \partial H \subseteq Z $ by \cite[7.3]{MR3466806}. Henceforth we have the following inclusions
\begin{equation*}
\partial H \subseteq	Z \subseteq \Tan(\Gamma,x) \subseteq \Tan(\partial A,x) \subseteq \Tan(A,x) \subseteq H
\end{equation*}
and one may infer from \cite[5.7]{MR3281655} that $ \Tan(A,x) = H $ and $ \Tan(\Gamma,x) = \partial H $.

Next we check that
\begin{equation*}
\partial A = \Gamma.
\end{equation*}
Let $ x \in \Gamma $. Then there exist an $ m $ dimensional plane $ T $, an open neighborhood $ W $ of $ x $ and a convex Lipschitzian function $ f : U \rightarrow T^{\perp} $ defined on a relatively open convex subset $ U $ of $ T $ containing $ T_{\natural}(x) $, such that
\begin{equation*}
W \cap \partial A = \{  \chi + f(\chi) : \chi \in U \}.
\end{equation*}
Since $ \Lip f < \infty $ it follows that $ \Tan(\partial A, y) \cap T^{\perp} = \{0\} $ for $ y \in W \cap \partial A $ and, since we have proved in the previous paragraph that $ \Tan(\Gamma,y) $ is an $ m $ dimensional plane for every $ y \in \Gamma $, we employ \cite[first paragraph p.\ 234]{MR0257325} to conclude
\begin{equation*}
T =	T_{\natural}(\Tan(\Gamma,y)) = \Tan(T_{\natural}(W \cap \Gamma), T_{\natural}(y)) \quad \textrm{for every $ y \in W \cap \Gamma $}.
\end{equation*}
Noting that $ T_{\natural}(W \cap \Gamma) $ is relatively closed in $ U $, we infer\footnote{\emph{Suppose $ C \subseteq U\subseteq \Real{n} $, $ U $ is open and $ C $ is relatively closed in $ U $. If $\Tan(C,x)= \Real{n} $ for every $ x \in C $ then $ C = U $.} In fact, if there was $ y \in U \sim C $ and if $ t = \sup\{s : \mathbf{U}(y,s)\cap C = \varnothing \} $ then $ t > 0 $, $ \mathbf{U}(y,t)\cap C = \varnothing $, $ \mathbf{B}(y,t) \cap C \neq \varnothing $ and $ y-x \in \Nor(C,x) $ for every $ x \in  \mathbf{B}(y,t) \cap C $. This is clearly a contradiction.} that $ T_{\natural}(W \cap \Gamma) = U $ and $ W \cap \partial A = W \cap \Gamma $. Since $ x $ is arbitrarily chosen in $ \Gamma $, it follows that $ \partial A = \Gamma $.

We combine the assertions of the previous two paragraphs with \cite[2.2.4]{MR3155183} to conclude that \emph{$\partial A $ is an $ m $ dimensional submanifold of class $ 1 $ in $ \Real{m+1} $.} Moreover, it is well known that $ \dmn \bm{\xi}_{A} = \Real{m+1} $, $ \Lip \bm{\xi}_{A}\leq 1 $ (see \cite[1.2]{MR3155183}) and $ \{ x : \bm{\delta}_{A}(x) < r  \} $ is an open convex set whose boundary $S(A,r)$ is an $ m $ dimensional submanifold of class $ \mathcal{C}^{1,1} $ for $ r > 0 $ (see \cite[4.8]{MR0110078}). Let $ 0 < r < m^{-1} $ and $ \xi = \bm{\xi}_{A}|S(A,r) $. For $ \Haus{m} $ a.e.\ $ x \in S(A,r) $ we apply the barrier principle \ref{weak maximum principle}, with $ T $, $ \eta $ and $ f $ replaced by $ \{ v : v \bullet \bm{\nu}_{A}(x)=0  \} $, $ \bm{\nu}_{A}(x) $ and a concave function whose graph corresponds to $S(A,r)$ in a neighborhood of $ x $, to infer (see \ref{curvature of arbitrary closed sets 3}) that
\begin{equation*}
\chi_{A,i}(x) \geq 0 \quad \textrm{for $ i = 1, \ldots , m $}, \qquad \sum_{i=1}^{m}	\chi_{A,i}(x) \leq m,
\end{equation*}
and we combine these inequalities to conclude that $\chi_{A,i}(x) \leq m $ for $ i = 1, \ldots , m $. Therefore $ \| \Der(\xi-\mathbf{1}_{S(A,r)})(x)\| \leq mr < 1 $ for $ \Haus{m} $ a.e.\ $ x \in S(A,r) $ and, noting that $ \xi $ is univalent by \cite[4.8(12)]{MR0110078}, we apply \ref{sufficient conditions for bilipschitzian homeomorphisms} to conclude that the function $ \xi^{-1} : \partial A \rightarrow S(A,r) $ is a locally Lipschtzian map and the unit normal vector field on $ \partial A $,
\begin{equation*}
\eta = \bm{\nu}_{A} \circ \xi^{-1},
\end{equation*}
is locally Lipschitzian. Combining \cite[3.25]{San} with \eqref{Area mean curvature characterization: eq 3} and \eqref{Area mean curvature characterization: eq 4}, we infer for $ \Haus{m} $ a.e.\ $ x \in \Gamma $ and for $ u,v \in \Tan(\Gamma, x) $ that
\begin{flalign*}
\Der \eta(x)(u)\bullet v & = -\ap \mathbf{b}_{\Gamma}(x)(u,v)\bullet \eta(x) \\
& = - m^{-1}(\ap \mathbf{h}_{\Gamma}(x)\bullet \eta(x))(u \bullet v) \\
& = u \bullet v,
\end{flalign*}
whence we conclude that $ \Der (\eta - \mathbf{1}_{\Gamma})(x) =0 $ for $ \Haus{m} $ a.e.\ $ x \in \Gamma $. Therefore there exists $ a \in \Real{m+1} $ such that
\begin{equation*}
	\eta(z) = z-a \quad \textrm{for every $ z \in \Gamma $}
\end{equation*}
and, since $ |\eta(z)| = 1 $ for $ z \in \Gamma $, we conclude that
\begin{equation*}
	\Gamma = \partial \mathbf{B}(a,1).
\end{equation*}
\end{proof}

\begin{Remark}\label{area mean curvature ch remark}
If $ V $ is a varifold as in \cite[Theorem 1]{MR855173} and if we additionally assume that $ V $ is \emph{integral} then Brakke perpendicularity theorem \cite[5.8]{MR485012} implies that $ |\mathbf{h}(V,x)|\leq m $, whence we deduce by \cite[2.8]{MR3466806} that $ \spt \|V\| $ is an $(m,m)$ subset of $ \Real{n} $. 
\end{Remark}

Theorem \ref{area mean curvature ch} readily provides a sufficient condition to conclude that the area-blow up set is empty for certain sequences of $ m $ dimensional varifolds whose mean curvature is uniformly bounded outside a set that is not too large.

\begin{Corollary}\label{area-blow-up set}
	Let $ V_{i} $ be a sequence of $ m $ dimensional varifolds in $ \Real{n} $ whose total variation $ \|\delta V_{i}\| $ is a Radon measure and such that the following three conditions hold for some $ 0 < h < \infty $:
	\begin{enumerate}
		\item the generalized boundaries of $ V_{i} $ are uniformly bounded on compacts sets; i.e.\ if $ \mu_{i} $ is the singular part of $ \| \delta V_{i} \| $ with respect to $ \|V_{i}\| $ then
		\begin{equation*}
			\limsup_{i \to \infty}\mu_{i}(K) < \infty \quad \textrm{for every compact set $ K \subseteq \Real{n} $;}
		\end{equation*}
		\item there exists a compact set $ \Gamma $ such that $ \Haus{m}(\Gamma) < (m/h)^{m} \Haus{m}(\mathbf{S}^{m}) $ and
		\begin{equation*}
		\limsup_{i \to \infty}\|V_{i}\|(K) < \infty \quad \textrm{for every compact set $ K \subseteq \Real{n} \sim \Gamma $;}
		\end{equation*}
		\item\label{area-blow-up set:3} $ \limsup_{i \to \infty} \int_{K} (|\mathbf{h}(V_{i},z)| - h  )^{+}\, d\|V_{i}\|z < \infty $ whenever $ K \subseteq \Real{n} $ is compact, where $ t^{+} = \sup\{t,0\} $ for $ t \in \Real{} $.
		
		Then $ \limsup_{i \to \infty} \|V_{i}\|(K) < \infty $ for every compact set $ K \subseteq \Real{n} $.
	\end{enumerate}
\end{Corollary} 

\begin{proof}
If $ Z = \{ x :  \limsup_{i \to \infty} \|V_{i}\|(\mathbf{B}(x,r)) = \infty \; \textrm{for every $ r > 0 $}\} $ then $ Z $ is an $ (m,h) $ subset of $ \Real{n+1} $ by \cite[2.6]{MR3466806}. Since $ Z \subseteq \Gamma $ and $ Z $ is compact, it follows from \ref{area mean curvature ch} that $ Z = \varnothing $.
\end{proof}

Here is the limit-case $ h =0 $.

\begin{Corollary}\label{area-blow-up set stationary}
	Suppose $ V_{i} $ is a sequence of $ m $ dimensional varifolds in $ \Real{n} $ such that
	\begin{enumerate}
	\item $ \limsup_{i \to \infty}\| \delta V_{i}\|(K) < \infty $ whenever $ K \subseteq \Real{n} $ is compact,
	\item there exists a compact set $ \Gamma \subseteq \Real{n} $ such that $ \Haus{m}(\Gamma) < \infty $ and 
	\begin{equation*}
	\limsup_{i \to \infty}\|V_{i}\|(K) < \infty \quad \textrm{for every compact set $ K \subseteq \Real{n} \sim \Gamma $.}
	\end{equation*}
	\end{enumerate}
	
Then $ \limsup_{i \to \infty} \|V_{i}\|(K) < \infty $ for every compact set $ K \subseteq \Real{n} $.
\end{Corollary}

\begin{proof}
Choose $ h > 0 $ small so that $ \Haus{m}(\Gamma) < (m/h)^{m}\Haus{m}(\mathbf{S}^{m}) $ and apply \ref{area-blow-up set}.
\end{proof}

\begin{Remark}
The reader may find useful to compare \ref{area-blow-up set} and \ref{area-blow-up set stationary} with \cite[1.4]{MR3466806}.
\end{Remark}


\medskip 

\noindent Institut f\"ur Mathematik, Universit\"at Augsburg, \newline Universit\"atsstr.\ 14, 86159, Augsburg, Germany,
\newline mario.santilli@math.uni-augsburg.de

\end{document}